\documentclass[a4paper,11pt]{amsart}
\usepackage[latin1]{inputenc}   
\usepackage{amssymb,amsmath,amsthm,mathrsfs}
\usepackage{graphicx,mathpazo}
\usepackage[all]{xy}
\usepackage[usenames,dvipsnames]{color}
\usepackage{enumitem}

\usepackage[colorlinks=true,linkcolor=Cerulean,pagebackref,citecolor=BurntOrange]{hyperref}

 \usepackage{tikz}
\usetikzlibrary{arrows,decorations.pathmorphing,decorations.pathreplacing,positioning,shapes.geometric,shapes.misc,decorations.markings,decorations.fractals,calc,patterns}

\tikzset{>=stealth',
     cvertex/.style={circle,draw=black,inner sep=1pt,outer sep=3pt},
     vertex/.style={circle,fill=black,inner sep=1pt,outer sep=3pt},
     star/.style={circle,fill=yellow,inner sep=0.75pt,outer sep=0.75pt},
     tvertex/.style={inner sep=1pt,font=\criptsize},
     gap/.style={inner sep=0.5pt,fill=white}}

\newcommand{\arrowrl}[3][20]
{
\hspace{-5pt}
\begin{tikzpicture}
\node (A) at (0,0) {};
\node (B) at (1,0) {};
\draw[->] ($(A)+(0,0.2)$) -- node [above] {$\scriptstyle f^*$} ($(B)+(0,0.2)$);
\draw [->] ($(B)+(0,0.2)$) -- node [below] {$\scriptstyle f_*$} ($(A)+(0,0.2)$);
\end{tikzpicture}
\hspace{-5pt}
}
\newcommand{\adj}[2][20]{\arrowrl}

\newcommand{\CC}{\ensuremath{\mathbb{C}}}
\newcommand{\QQ}{\ensuremath{\mathbb{Q}}}
\newcommand{\NN}{\ensuremath{\mathbb{N}}}

\newcommand{\ra}{\ensuremath{\rightarrow}}
\newcommand{\lra}{\ensuremath{\longrightarrow}}

\newcommand{\tr}{\ensuremath{\tau}} 

\DeclareMathOperator{\add}{add}

\DeclareMathOperator{\depth}{depth}

\DeclareMathOperator{\End}{End}
\DeclareMathOperator{\Ext}{Ext}

\DeclareMathOperator{\Hom}{Hom}

\DeclareMathOperator{\op}{op}

\DeclareMathOperator{\Sing}{Sing}
\DeclareMathOperator{\Spec}{Spec}

\newcommand{\mc}[1]{\ensuremath{\mathcal{#1}}}   
\newcommand{\mf}[1]{\ensuremath{\mathfrak{#1}}}   

\newcommand{\CM}{\operatorname{{CM}}}
\newcommand{\gl}{\operatorname{gl.dim}\nolimits}
\newcommand{\gs}{\operatorname{gs}\nolimits}

\newcommand{\mmod}[1]{\operatorname{mod}(#1)}

\newcommand{\ef}[1]{{\color{black} #1}}

\theoremstyle{theorem}
\newtheorem{Thm}{Theorem}[section]         
\newtheorem{lem}[Thm]{Lemma}
\newtheorem{cor}[Thm]{Corollary}

\newtheorem{Prop}[Thm]{Proposition}   

\newtheorem{Qu}[Thm]{Question}

\theoremstyle{remark}
\newtheorem{Bem}[Thm]{Remark}
\newtheorem{ex}[Thm]{Example}

\theoremstyle{definition}
\newtheorem{defi}[Thm]{Definition} 

\title{Trace ideals, normalization chains, and endomorphism rings}

\author{Eleonore Faber}

\address{
School of Mathematics, University of Leeds, LS2 9JT Leeds, UK
}

\email{e.m.faber@leeds.ac.uk}

\date{\today}
\thanks{
\noindent The author is a Marie Sk{\l}odowska-Curie fellow at the University of Leeds (funded by the European Union's Horizon 2020 research and innovation programme under the Marie Sk{\l}odowska-Curie grant agreement No 789580).
} 

\begin{document}

\begin{abstract} 
In this paper we consider reduced (non-normal) commutative noetherian rings $R$. With the help of conductor ideals and trace ideals of certain $R$-modules we deduce a criterion for a reflexive $R$-module to be closed under multiplication with scalars in an integral extension of $R$. Using results of Greuel and Kn\"orrer this yields a characterization of plane curves of finite Cohen--Macaulay type in terms of trace ideals. \\
Further, we study one-dimensional local rings $(R,\mf{m})$ such that that their normalization is isomorphic to  the endomorphism ring $\End_R(\mf{m})$: we give a criterion for this property in terms of the conductor ideal, and show that these rings are nearly Gorenstein. 
Moreover, using Grauert--Remmert normalization chains, we show the existence of noncommutative resolutions of singularities of low global dimensions for curve singularities.

\end{abstract}

 \dedicatory{Dedicated to Gert-Martin Greuel on the occasion of his 75th birthday}

\maketitle


\section{Introduction }

This paper studies trace ideals and conductor ideals of reduced commutative noetherian rings and their relation to endomorphism rings of modules, in particular maximal Cohen-Macaulay modules and reflexive modules. The motivation comes from two directions: on the one hand, in order to compute the integral closure of a commutative ring $R$ in its total quotient ring (i.e., the \emph{normalization} $\widetilde R$ of $R$), one forms an ascending chain of endomorphism rings of certain ideals that stabilizes at the normalization (first studied in the analytic context by Grauert and Remmert \cite{GrauertRemmert71}). This can be used to formulate an algorithm for computation of integral closure, which has been implemented in computer algebra systems like {\sc Singular}. See \cite{deJong98, GreuelPfister08} for the original algorithms, and \cite{BoehmDeckerSchulze,GreuelLaplagneSeelisch,ParallelNormalization} for more recent enhancements of these ideas. \\
 Here, one would like to have a short chain of endomorphism rings, so that the normalization is reached in few steps of the algorithm. If one knew the conductor ideal $\mc{C}$ of the normalization, then this chain would only be of length $1$, since $\End_R(\mc{C})$ is the normalization of $R$. In general it is not possible to find the conductor ideal, and thus we have to come up with some test ideals that are hopefully sufficiently close to the conductor. Therefore we will study \emph{trace ideals} of $R$-modules: these ideals are easy to calculate from the presentation matrix of the module and we will see that they tell us whether the module from which they come from is closed under multiplication with scalars of a certain integral extension of $R$. This also leads us to consider conductor ideals of smaller integral extensions of $R$ than $\widetilde R$. \\
In the study of trace ideals it is  also natural to ask which ideals in a commutative ring $R$ can occur as trace ideals. Recently it has been studied, when \emph{every} ideal of $R$ is isomorphic to a trace ideal \cite{Lindo, LindoPande}. See  \cite{KobayashiTakahashi} for an answer in the local case and \cite{GotoIsobeKumashiro} for connections with stable ideals. Since here we are mainly interested in the class of Cohen--Macaulay modules, one is led to the finer question: which ideals in $R$ are isomorphic to trace ideals of Cohen--Macaulay modules over $R$? The first case to consider, is $R$ of \emph{finite $\CM$-type}, that is, there are only finitely many isomorphism classes of Cohen--Macaulay modules over $R$. \\ 
Our main results in this direction are: a criterion for a reflexive module over a ring to be a module over an integral extension using trace ideals and conductor ideals (Theorem \ref{Thm:traceequalconductor}), and a characterization of plane curves of finite Cohen--Macaulay type with trace ideals (Cor.~\ref{Cor:finitetraceCM}). \\

As another application of normalization chains, we are interested in rings that are ``nearly'' normal from the point of view of the normalization algorithm: here we study one dimensional reduced local rings $(R,\mf{m})$. If such an $R$ is not regular, then the singular locus of $\Spec(R)$ is zero-dimensional and determined by the maximal ideal of $R$. In this case there is a natural chain of endomorphism rings, cf.~\cite{Leuschke07, IyamaRejective}, starting with
$$R \subseteq \End_R(\mf{m}) \subseteq \cdots $$
We say that $R$ has a \emph{$1$-step normalization} or is \emph{$1$-step normal} if $\widetilde R \cong \End_R(\mf{m})$. Note that if $R$ is regular, then it is also $1$-step normal, since then $R$ is isomorphic to its maximal ideal.\\

 We give a characterization of $1$-step normal rings in terms of the maximal ideal (Prop.~\ref{Prop:conductor-1-step} and Cor.~\ref{Cor:maximal-isom-normalization}) and study connections with nearly and almost Gorenstein rings that have recently appeared in work of Herzog--Hibi--Stamate  about the trace of canonical modules \cite{HerzogHibietc}. \\ 
 
The other direction of research is the study of endomorphism rings of finitely generated modules over a commutative ring. These endomorphism rings are in general noncommutative but still inherit some nice properties (like noetherianity) from $R$. Recently, the study of various endomorphism rings has flourished in both commutative and noncommutative algebra, as well as in algebraic geometry and representation theory and even has applications to theoretical physics. In particular interesting are endomorphism rings of modules that have \emph{finite global dimension}. They can be seen as a noncommutative analog of resolution of singularities: let $R$ be a reduced noetherian ring and let $M$ be a faithful $R$-module. Then $\Lambda=\End_R(M)$ is called a \emph{noncommutative resolution (=NCR) of singularities} if $\gl \End_R(M) < \infty$, see \cite{DaoIyamaTakahashiVial}. Moreover, if $\Lambda$ is homologically homogeneous, then $\Lambda$ is called a \emph{noncommutative crepant resolution (=NCCR) of singularities}, cf.~\cite{vandenBergh04,IyamaWemyss10}. For more about the rationale behind these definitions, see \cite{Leuschke12, DFI}. \\
Recently, there were quite a few results on construction of NC(C)Rs and their properties, see e.g. \cite{IyamaWemyss10a, BLvdB10, FMS, IyamaNakajima, HibiRings, SpenkovdB}. In particular, it is interesting which values the global dimension can assume: this should give some information about the singularities of $\Spec(R)$. Therefore, in \cite{DFI} the \emph{global spectrum $\gs(R)$ of a singularity $\Spec(R)$} was defined as the set of all possible finite global dimensions of endomorphism rings of Cohen--Macaulay R-modules. 
In \cite[Theorem 4.6]{DoFI} the global spectra of some ADE-curve singularities were determined. \\

Our main result in this direction is that for \emph{any} curve singularity, the integers $1,2$ are in the global spectrum (Lemma \ref{12ingspec}) and $3\in \gs(R)$ if and only if the singularity $\Spec(R)$ is not of type $A_{2n}$ (Theorem \ref{Thm:3ingspec}).

\subsection{Structure of the paper} 
Our goal was to make this paper as self-contained as possible, thus in Section \ref{Sec:Prelim} we first review some homological facts and then introduce the main players: trace ideals and conductor ideals. We also review the construction of the normalization of a commutative ring by an ascending chain of endomorphism rings (the Grauert--Remmert normalization algorithm). In the next section we characterize reflexive $R$-modules, where $R$ is a reduced ring, that are closed under scalar multiplication with elements in a finite birational extension $R'$, that is, these $R$-modules are also $R'$-modules (Theorem \ref{Thm:traceequalconductor}). In Section \ref{Sub:traceandfiniteCM} the relationship between finite Cohen--Macaulay type and trace ideals is studied, in particular, we show that the coordinate ring of a (reduced) plane curve singularity is of finite CM-type if and only if there are finitely many possibilities for trace ideals of $\CM$-modules over this ring (see Cor.~\ref{Cor:finitetraceCM}). 

In Section \ref{Sec:one-step-normal} we consider reduced one-dimensional local rings (``curves'') and study those with $1$-step normalization, that is, $(R, \mf{m})$ such that $\widetilde R \cong \End_R(\mf{m}$). Using the reflexivity of the maximal ideal $\mf{m}$ (Prop.~\ref{Prop:mreflexive}), we deduce that these $1$-step normal rings are characterized by the property $\mf{m} \subseteq \mc{C}$, the conductor of the normalization (Prop.~\ref{Prop:conductor-1-step}). This implies that $1$-step normal rings are nearly Gorenstein (Cor.~\ref{Cor:1-step-nearly-Gorenstein}) in the sense of \cite{HerzogHibietc}. We also show that for $1$-step normal rings the maximal ideal $\mf{m}$ is isomorphic to its dual $\mf{m}^*=\Hom_R(\mf{m},R)$ (Prop.~\ref{Prop:onestepm}\footnote{Here the referee pointed out a much shorter proof of this proposition than in the first version of this manuscript.}), but that this property does not characterize $1$-step normal rings (Example \ref{Ex:selfdualmax}). 

In the final section we consider the global spectrum $\gs(R)$ of curve singularities $\Spec(R)$.  Making use of normalization chains and methods from representation theory, we show that $1,2$ and $3$ are contained in the global spectrum if and only if $R$ is not the coordinate ring of an $A_{2n}$-singularity (see Thm.~\ref{Thm:3ingspec}).

\section{Preliminaries} \label{Sec:Prelim}

\subsection{Conventions} \label{Sub:conventions}
In this paper, $\Lambda$ will denote any ring and the letter $R$ is reserved for a commutative noetherian ring. Additional assumptions, such as local, complete, Cohen--Macaulay or Gorenstein, will be explicitly stated. 

Recall that a commutative ring $R$ is \emph{reduced} if it has no non-zero nilpotent elements, or equivalently, it satisfies Serre's conditions $(R_0)$ and $(S_1)$, see \cite{Huneke06}. Also recall that the normalization $\widetilde R$ (=integral closure of $R$ in its total ring of fractions) is finitely generated as $R$-module if and only if the conductor $\mathcal{C}_{\widetilde R/R}$ of the normalization $\widetilde R$ in $R$ contains a non-zerodivisor. 

We are studying singularities of $\Spec(R)$, so most of the time we will assume that $R$ is reduced. In the study of normalizations we will also assume that the normalization $\widetilde R$ is a finitely generated $R$-module.

\subsection{Homological properties of modules}

In the following, maximal Cohen--Macaulay modules over $R$ will play a significant role. These modules have been studied not only in commutative algebra but are also important in representation theory. In particular the description of maximal Cohen--Macaulay modules over hypersurface rings via matrix factorizations is very useful, see e.g. \cite{LeuschkeWiegand, Yos, BuchweitzMCM}. Here we give the most general definition for not necessarily commutative rings, as well as the standard definitions in the commutative algebra context, for a good reference see e.g.~\cite{BrunsHerzog93}.

Let $\Lambda$ be an Iwanaga--Gorenstein ring. Recall, that this means that $\Lambda$ is noetherian and has finite injective dimension as left as well as right module over itself. Then a finitely generated $\Lambda$-module $M$ is called \emph{(maximal) Cohen--Macaulay}, abbreviated $\CM$, if $\Ext^i_\Lambda(M,\Lambda)=0$ for all $i >0$.  If $(R, \mf{m},k)$ is any local commutative ring, then a finitely generated $R$-module $M$ is $\CM$ if its depth  is equal to the (Krull-)dimension $\dim(R)$ of $R$, that is, the smallest $d \geq 0$ for which $\Ext^d(k,M) \neq 0$, is equal to $\dim(R)$.  If $R$ is any commutative ring, then a finitely generated $R$-module $M$ is $\CM$ if for any maximal ideal $\mf{m}$ of $R$ the localization $M_\mf{m}$ is $\CM$ in $R_\mf{m}$. 
 If $R$ is any commutative ring, we denote by $\mmod{R}$ the category of finitely generated $R$-modules.  The full subcategory of $\CM$-modules of $\mmod{R}$ is denoted by $\CM(R)$. Moreover, $R$ is called a \emph{Cohen--Macaulay ring} if it is $\CM$ as a module over itself.

\begin{Prop}[\cite{BuchweitzMCM}, Lemma 4.2.2.(iii)] \label{Prop:CMimpliesReflexive}
Let $\Lambda$ be an Iwanaga--Gorenstein ring (not necessarily commutative) and let $M$ be a Cohen--Macaulay module over $\Lambda$. Then $M$ is reflexive and $M^*=\Hom_\Lambda(M,\Lambda)$ is also Cohen--Macaulay over $\Lambda^{op}$. In particular, if $\Lambda$ is commutative, then $M^*$ is in $\CM(\Lambda)$.
\end{Prop}

\begin{Prop}[\cite{Bass63}, Proposition 6.1]  \label{Prop:BassMstar-reflexive}
Let $R$ be a local, reduced Cohen--Macaulay ring and $M$ an $R$-module. Then the dual module $M^*=\Hom_R(M,R)$ is a reflexive $R$-module.
\end{Prop}

\begin{Prop}[\cite{ChristensenGdim}, Prop. (1.1.9)] \label{Prop:reflexivebidual}
Let $R$ be a commutative noetherian ring and $M$ be a finitely generated module. Then $M$ is reflexive if and only if $M$ is isomorphic to its bidual $M^{**}=\Hom_R(M^*,R)$.
\end{Prop}

The following facts about modules of homomorphism are helpful for computations with these modules and will be used later:

\begin{Prop}[\cite{HerzogKunz}, Lemma 2.1] \label{Prop:quotientHom} Let $R$ be a commutative ring, $K=Q(R)$ its total ring of quotients and $I,J \in K$ be two fractional ideals, such that $\depth_R(I)$ and $\depth_R(J)$ are greater or equal to $1$ (that is, both ideals contain a non-zerodivisor). Then $(I:_KJ)=\{ x \in K: xJ \subseteq I \}$ is isomorphic to $\Hom_R(J,I)$ as $R$-modules, via the homomorphism
$$\varphi: (I:J) \xrightarrow{} \Hom_R(J,I): x \mapsto (m \mapsto mx) \ .$$
\end{Prop}

\begin{Bem}
\begin{enumerate}[leftmargin=*]
 \item In particular, if $I=R$ and $J \supseteq R$ in Prop.~\ref{Prop:quotientHom}, then $(I:_KJ)=(I:_RJ)$, since $1_R \in J$.
 \item  In the following, we will be in particular interested in endomorphism rings $\End_R(I)$. If $R$ is a reduced noetherian ring, and $I \subseteq R$ and ideal containing a non-zerodivisor $x$ of $R$, then it is easy to see (e.g., in \cite[Lemma 3.6.1]{GreuelPfister08}), that $\Hom_R(I,I) \cong \frac{1}{x}(xI:I)$ as rings.
\end{enumerate}
\end{Bem}

The following lemma deals with change of the base ring and homomorphism modules, and often allows one to consider endomorphism rings of generators. Here $M$ is a \emph{generator} (of $\mmod{R}$) if every finitely generated $R$-module is a homomorphic image of a direct sum of copies of $M$. Equivalently, $R$ is a direct summand of $M^n$ for some positive integer $n$.

\begin{lem} \label{Lem:HomChangeofRings}
Let $R$ be a reduced local ring. Suppose that $S$ is a finite birational extension of $R$, i.e., $R \subseteq S \subseteq \widetilde R$ and $S$ is finitely generated as $R$-module. Then if $M$ and $N$ are $(R,S)$-bimodules and $N$ is torsion-free, one has 
$$\Hom_R(M,N) = \Hom_S(M,N).$$ 
\end{lem}

\begin{proof}
See \cite{LeuschkeWiegand}, Prop. 4.14.
\end{proof}

\begin{Bem}
 Assume that $R, S, M, N$ are as in the lemma. If $R$ is of dimension $1$, then the lemma holds for any Cohen--Macaulay module $N$, since torsion-free is equivalent to Cohen--Macaulay under this assumption on $R$. If $R$ is Gorenstein of any dimension $>0$, then the above lemma also holds for $N \in \CM(R)$, since then $N$ is reflexive and reflexive implies torsion-free (see e.g. \cite[Exercise 1.4.19]{BrunsHerzog93}). Finally if $R$ is a local domain of any dimension $>0$, then the above lemma also holds for $N \in \CM(R)$, since the depth of a torsion-free module is $\geq 1$. 
 \end{Bem}

\subsection{Trace ideals} 
Here we will give the general definition for trace ideals over any ring $\Lambda$. Later we will only study them for commutative rings $R$. 

Let $\Lambda$ be a ring, $M$ a right $\Lambda$--module. We set $M^{*} = \Hom_{\Lambda}(M,\Lambda)$, the $\Lambda$--dual of $M$ endowed with its
natural structure of a left $\Lambda$--module. Note that $M$ is as well a left $E=\End_{\Lambda}(M)$, right $\Lambda$--bimodule and
$M^{*}$ a left $\Lambda$, right $\End_{\Lambda}(M)$--bimodule. 
Here we view $M$ as an $\Lambda\otimes E^{\op}$--module and
$M^{*}$ as an $\Lambda^{\op}\otimes E$--module.
The tensor product here can be taken over any subring
in the centre $Z(\Lambda)$ of $\Lambda$, as such a subring then also maps to the centre of $E$.

The natural pairing
\begin{align*}
M^{*}\times M&\rightarrow \Lambda
\\
(\lambda,m)&\mapsto \lambda(m)
\end{align*}
satisfies $\lambda(\varphi(m))=(\lambda\circ\varphi)(m)$ for each $\Lambda$--endomorphism $\varphi \colon M\to M$ and thus
induces an $\Lambda$--bilinear {\em trace homomorphism\/}
\begin{align*}
\tr_{\Lambda}\colon M^{*}\otimes_{E}M\to \Lambda\quad, \quad\lambda\otimes m\mapsto \lambda(m)
\end{align*}
The image $I=\tr_{\Lambda}(M)$ is a two--sided ideal in $\Lambda$, the {\em trace ideal\/} of $M$. We sometimes denote it simply by $\tr(M)$, when it is clear in which ring we are working.

For computation of trace ideals, we note the following: 
If $R$ is a commutative Cohen--Macaulay ring and $I$ is an ideal of grade $\geq 1$, then $\tr(I)$ is the fractional ideal
$$ \tr(I)=I \cdot I^{-1}\ ,$$
where $I^{-1}$ is defined to be the quotient $(R:_{Q(R)}I)$, which is in this case isomorphic to $I^*=\Hom_R(I,R)$. For a proof, see \cite[Lemma 1.1]{HerzogHibietc}. In particular, if $I$ is an ideal of grade $\geq 2$ on $R$, then $\tr(I)=I$, see \cite[Example 2.4]{Lindo}. \\
If $M$ is a $\CM$-module over $R$ and $R$ is a hypersurface ring of the form $A/(f)$, where $A=k[[x_1, \ldots, x_n]]$ or $A=k[x_1, \ldots, x_n]$, then $M$ can be represented by a matrix factorization $(X,Y)$ of $f$. Then $\tr(M)=I_1(\mathrm{syz}(M))$, the first fitting ideal of of the syzygy module of $M$, see \cite{Vasconcelos98}. Here $I_1(\mathrm{syz}(M))$ is the ideal generated by the entries of the matrix $Y$.

Some general facts about trace ideals of commutative rings $R$:

\begin{lem}[cf.~Prop.~2.8 in \cite{Lindo}]  \label{Lem:traceprops}
We have the following properties of trace ideals of finitely presented $R$-modules $M$, $N$: 
\begin{enumerate}[leftmargin=*]
\item $\tr(M \oplus N)=\tr(M) + \tr(N)$.
\item Let $R'$ be a commutative finitely generated flat $R$-algebra. Then $\tr_{R'}(M\otimes_R R')=\tr_R(M)\otimes_R R'$. This implies in particular, that taking trace ideals commutes with completion and localization.
\item If $M$ is reflexive, then $\End_R(M)\cong \End_R(M^*)$ and $\tr(M)=\tr(M^*)$. 
\item $\tr(M)=R$ if and only if $M$ is a generator of $\mmod{R}$. Note that if $R$ is local, this just means that $R \in \add(M)$.
\item $\tr(M)=\tr(M/\mathrm{tors}_R(M))$, where $\mathrm{tors}_R(M)$ denotes the torsion submodule of $M$.
\end{enumerate}
\end{lem}

\begin{proof}
(1)--(4) are proven in Prop.~2.8 in \cite{Lindo}.\\
We prove (5) for a lack of reference: recall that 
$$\mathrm{tors}_R(M)=\{ m \in M: \text{ there exists a non-zerodivisor  }r \in R \textrm{ such that } rm=0\}.$$
 If $f: M \ra R$ is a homomorphism, then for any $m \in \mathrm{tors}_R(M)$ it follows that $f(m)=0$. 
Thus, for any $g(m)$ for $g \in \Hom_R(M,R)$ define $\bar g: M/\mathrm{tors}_R(M) \xrightarrow{} R, \bar m \mapsto g(m)$. It is easy to see that $\bar g$ is well defined. It follows that $g(m)=\bar g(\bar m)$ for any $m \in M$ and thus $\tr(M) \subseteq \tr(M/\mathrm{tors}_R(M)$. 
On the other hand, for any $\bar f \in \Hom_R(M/\mathrm{tors}(M),R)$ one can define $f \in \Hom_R(M,R)$ such that $f(m):=\bar f(\overline{m})$, where $\overline{m}$ is the image of $m$ under the natural projection $M \xrightarrow{} M/\mathrm{tors}_R(M)$. So for any $\bar f(\overline{m})$ there exists a homomorphism $f: M \ra R$ such that $f(m)=\bar f(\overline{m})$, which shows $\tr (M) \supseteq \tr(M/\mathrm{tors}_R(M))$. 
\end{proof}

\begin{Bem}
In (3), one may ask, whether either $\End_R(M) \cong \End_R(M^*)$ or $\tr(M)=\tr(M^*)$ characterizes reflexive modules. This is not the case: for the first property take $R=k[[x,y]]$ and $M=(x,y)$. Then $\End_R M = R$ and since $M^*=R$ also $\End_RM^*=R$, but clearly $M$ is not reflexive (its depth is $1$). For the second property take any local ring $(R, \mf{m})$ and $M=R/\mf{m} \oplus R$. Then by (1) of the Proposition, $\tr(M)=R$ and since $M^*\cong R$ 
$\tr (M^*)=R$. But clearly $M$ is not reflexive, since $M^{**}=R$.
\end{Bem}

\subsection{Conductors} 

Next we collect some properties of conductor ideals in order to study the relationship of conductor ideals and trace ideals in Section \ref{Sec:TraceConductor}. 

If $R$ is commutative noetherian and reduced, then recall that the \emph{normalization} $\widetilde R$ of $R$ is the integral closure of $R$ in its total ring of quotients $Q(R)$. If $R = \widetilde R$, then $R$ is called \emph{normal}. As already mentioned in Section \ref{Sub:conventions}, we will assume that $\widetilde R$ is module-finite over $R$. In particular, if $R$ is $1$-dimensional and local, then the normalization is $R$-module finite if $R$ is analytically unramified (equivalently, the completion $\widehat R$ is reduced), see \cite[Kor.~2.12]{HerzogKunz}. 

Following \cite{LeuschkeWiegand, GreuelKnoerrer}, recall that $R \subseteq S$ is a \emph{finite birational extension of rings} if $S$ is a ring contained in the total quotient ring $Q(R)$ and is finitely generated as an $R$-module. 
Note that this implies that $S$ is an integral extension of $R$ (see \cite[Appendix A]{BrunsHerzog93}).

\begin{defi} 
Let $R$ be a reduced ring and let $R \subseteq R'$ be a reduced extension of $R$ such that $R'$ is contained in $Q(R)$, the total ring of fractions, and moreover $R'$ is module finite over $R$. Then $(R:_{Q(R)}R')=\{ a \in Q(R): aR' \subseteq R\}$ is called the \emph{(relative) conductor of $R'$ in $R$}, denoted $\mc{C}_{R'/R}$. For $R'=\widetilde R$, the normalization of $R$, one simply calls $\mc{C}:=(R:_{Q(R)}\widetilde R)$, the \emph{conductor of $R$}.
\end{defi}

It is easy to see that $(R:_{Q(R)}R')=(R:_RR')$ is an ideal in $R$ and $R'$, and that $\mc{C}_{R'/R}=\Hom_R(R',R)=\mathrm{Ann}(R'/R)$.  
In fact, $\mc{C}_{R'/R}$ is the largest common ideal of $R'$ and $R$. Note that in particular, if $R=R'$, then $\mc{C}_{R'/R}=R$. %

\begin{lem}
Let $R$ be  reduced and let $R'$ be a reduced extension with $R' \subseteq Q(R)$ and $R'$ module finite over $R$. Then
$$R \subseteq R' \subseteq  \Hom_R(\mc{C}_{R'/R},R) \ . $$
\end{lem}

\begin{proof}  
From Prop.~\ref{Prop:quotientHom} it follows that $\Hom_R(\mc{C}_{R'/R},R)=(R: \mc{C}_{R'/R})$ contains $R$. Since $\Hom_R(\mc{C}_{R'/R},R)=(R')^{**}$, it follows that $R'\subseteq \Hom_R(\mc{C}_{R'/R},R)$. 
\end{proof}

The relation between conductor ideals and reflexive finite birational extensions of $R$ is described by the following result. 

\begin{Thm}[\cite{DesalvoManaresi}, Theorem 1.4] \label{Thm:Conductors-Overrings} Let $R$ be a reduced ring and let $\mc{F}$ be the set of all reflexive finite birational extensions of $R$. Then the map 
$$\mc{F} \xrightarrow{} \{ I \subseteq R: I \textrm{ is a conductor of an element in } \mc{F} \}: S \mapsto \mc{C}_{S/R} $$
is an order inverting bijection between the elements in $\mc{F}$ and their conductors. The inverse map is is given by $I \mapsto \End_R(I)$.
\end{Thm}

\begin{lem} \label{Lem:ConductorReflexive}
Let $R$ be a reduced ring. Then the normalization $\widetilde R$ of $R$ and the conductor of the normalization $\mc{C}$ are both reflexive $R$-modules. Moreover, one has
$$\End_R(\mc{C})=\widetilde R \ .$$
\end{lem}

\begin{proof}
The equality $\End_R(\mc{C})=\widetilde R$ is shown in \cite[Remark after Prop.~1.2]{DesalvoManaresi}, which also shows the reflexivity of $\widetilde R$. Alternatively, in view of Prop.~\ref{Prop:reflexivebidual} one can show, using adjunction, that $\widetilde R \cong \mc{C}^*$ and thus $\widetilde R$ is reflexive. Dualizing again yields that $\mc{C}$ is reflexive. 
\end{proof}

\subsection{Normalization chains} \label{Sub:normalizationchains}
Here we briefly describe the ideas of the Grauert-Remmert normalization algorithm, which was our main motivation to study endomorphism rings of trace ideals. 

Let $R$ be a reduced noetherian (commutative) ring. One calls $N(R):=\{ \mf{p} \in \Spec(R): R_\mf{p}$ is not normal$\}$ the non-normal locus of $R$. It can be shown (see e.g. \cite[Lemma 3.6.3]{GreuelPfister08}) that $N(R)=V(\mc{C}_{\widetilde R / R})$, that is, the conductor of the normalization defines precisely the non-normal points of $\Spec(R)$. 

If $I$ is an ideal of $R$ that contains a non-zerodivisor, then 
$$ R \subseteq \End_R(I) \subseteq \widetilde R \ . $$
If $I = \mc{C}$, then $\End_R \mc{C} = \widetilde R$ (see Lemma \ref{Lem:ConductorReflexive}). But in practice it is hard to guess the conductor ideal, and thus one needs to proceed in steps. 

The rough idea is (here we follow \cite{deJong98} and \cite[3.6]{GreuelPfister08}): Start with the reduced non-normal ring $R_0:=R$ and an ideal $I_0 \subseteq R_0$ containing a non-zerodivisor and $\mathrm{supp}(I_0)=\Sing(\Spec(R_0))$ such that $R_0 \subsetneq \End_{R_0}(I_0)$ (this holds e.g. for radical ideals with $V(I_0) \supseteq N(R_0$)). Then $R_1:=\End_{R_0}(I_0)$ is a reduced ring lying between $R$ and $\widetilde R$. If $R_1$ is equal to the normalization, we are finished, otherwise find an ideal $I_1 \subseteq R_1$ such that $R_1 \subsetneq R_2:=\End_{R_1}(I_1)$ and repeat if necessary. This yields a chain of rings
\begin{equation}  \label{Eq:chain}
 R_0=R \subsetneq R_1=\End_{R_0}(I_0) \subsetneq \cdots \subsetneq R_l=\End_{R_{l-1}}(I_{l-1}) \ ,
\end{equation}
with $\mathrm{supp}(I_i) \subseteq \Sing(\Spec(R_i))$ which terminates at $R_{l}= \widetilde R$, $l < \infty$, if $\widetilde R$ is a finitely generated $R$-module. We call such a chain of rings a \emph{(Grauert--Remmert) normalization chain of $R$}. The crucial point in this algorithm is to find good \emph{test ideals} $I_j$ that are easy to compute, so that one needs as few steps as possible in the algorithm. In practice one takes $I_j$ to be the radical of the Jacobian ideal of $R_j$, see \cite[Algorithm 3.6.9]{GreuelPfister08}. However, it would be interesting to consider other type of ideals, or even endomorphism rings of modules of higher rank. In the next section, we will therefore study trace ideals, which are easy to calculate and sometimes are closer to the conductor than the radical of the Jacobian ideal.

\section{Traces and conductors} \label{Sec:TraceConductor}

Here we study connections between trace ideals of certain $R$-modules and conductor ideals of finite birational extensions of $R$. In particular we are interested in trace ideals of $\CM$-modules and we consider the finite $\CM$-type case in \ref{Sub:traceandfiniteCM}. \\
Note that the relation between trace ideals and birational extensions of $R$ has in particular been studied in \cite[Prop.~1.2]{GotoIsobeKumashiro}

\begin{Prop} \label{Prop:traceconductor}
Let $R \subseteq R'$ be a finite birational extension  of a reduced commutative noetherian ring $R$ and let $M'$ be a module over $R'$. Then
$$\tr_R(M') \subseteq \mc{C}_{R'/R}.$$
In particular, if $M' \cong R'$, then also the other inclusion holds, that is, 
$$\tr_R(R') = \mc{C}_{R'/R}.$$
\end{Prop}

\begin{proof}
Clearly $M'$ is also a module over $R$, so $\tr_R(M')$ is well-defined. Since we can identify $\mc{C}_{R'/R}$ with $(R:R')=\{ a \in R \mid aR' \subseteq R\}$, we have to show that for any $R$-linear map $f$ from $M'$ to $R$, for any $m \in M'$ and any $x \in R'$, the element $f(m) x$ is contained in $R$. An element $x\in R'$ can be written as $x=\frac{r}{s}$, for some $r, s \in R$, where $s$ is a non-zerodivisor. Then
$$sf(xm)=f(rm)=rf(m)= sxf(m),$$
and since $s$ is a non-zerodivisor, it follows that $f(xm)=xf(m)$. 
Since $M'$ is an $R'$-module, $xm \in M'$ and since $f:M' \xrightarrow{} R$ is an $R$-linear map, $f(xm)$ lies in $R$. This shows that $\tr_R(M') \subseteq \mc{C}_{R'/R}$. \\
For the second assertion it has to be shown that $\mc{C}_{R'/R} \subseteq \tr(R')$. Therefore let $a \in \mc{C}_{R'/R}$. Define a map $f_a: R' \ra R$ as $x \mapsto xa$. Then $f_a$ is well-defined since by definition $xa \in R$ for any $x \in R'$ and of course $f_a$ is $R$-linear. Clearly $f_a(1)=a$, so $a \in \tr_R(R')=(f(x) \mid f \in \Hom_{R}(R',R), x \in R')$.
\end{proof}

\begin{lem} Let $R \subseteq R' \subseteq R^{''}$ be finite birational extensions of a commutative noetherian ring $R$. Then: \\
(1) $\mc{C}_{R''/R} \subseteq \mc{C}_{R'/R}$. \\
(2) $\tr(R^{''}) \subseteq \tr(R')$. 
\end{lem}

\begin{proof}
(1): clear from the definitions of the relative conductors: 
$$\mc{C}_{R''/R}=\{ a \in R: aR'' \subseteq R\} \subseteq \{b \in R: bR' \subseteq R\}=\mc{C}_{R'/R},$$
 since $R' \subseteq R''$. \\
(2): Follows from (1) and Prop.~\ref{Prop:traceconductor}.
\end{proof}

\begin{ex}
In general one only has $\tr(M) \subseteq \mc{C}_{R'/R}$ for a $R'$-module $M$ and not equality: consider e.g. $R=\CC[[x,y]]/(x^3+y^4)$ the $E_6$ singularity. Then one can compute the indecomposable $\CM-R$-modules, see \cite{Yos}. The ring $R'=\End_R(\mf{m}_R)$ is a birational extension of $R$ and has the same indecomposable $\CM$-modules as $R$, with exception of $R$ itself. From Yoshino's list (see loc.~cit.~) of matrix factorizations of the indecomposables (in Yoshino's notation) one sees that $\tr (M_2)=(x,y^2)$ is strictly contained in $\mc{C}_{R'/R}=\mf{m}_R$. 
\end{ex}

There is a converse for reflexive modules:

\begin{Prop} \label{Prop:tracemodule}
Let $R \subseteq R'$ be a finite birational extension of a commutative noetherian reduced ring $R$ and let $M$ be a module over $R$. If $\tr_R(M) \subseteq \mc{C}_{R'/R}$, then $M^*$ is a module over $R'$. If moreover $M$ is reflexive, then $M$ is a module over $R'$. \end{Prop}

\begin{proof}
Let $f \in M^*$ and $x \in R'$. Then $x \cdot f$ is a function from $M \ra R$, since for any $m \in M$ one has $x f(m) \in R$, which follows from the containment of $\tr_R(M)$ in the conductor $\mc{C}_{R'/R}$. Clearly $xf$ is an $R$-linear function, thus $M^*$ is also an $R'$-module.\\
By Prop. \ref{Prop:traceconductor}, this implies that $\tr_R(M^*) \subseteq \mc{C}_{R'/R}$. 
 Thus, by what we just proved, $(M^*)^*=M^{**}$ is also an $R'$-module. If $M$ is reflexive, this means that $M$ is a module over $R'$. 
\end{proof}

\begin{Thm} \label{Thm:traceequalconductor}
Let $R$ be a reduced noetherian local ring and let $M$ be a finitely generated reflexive $R$-module and let $R'$ be a finite birational extension. The following two assertions are equivalent: \\
(1) $\tr(M)\subseteq \mathcal{C}_{R'/R}$, where $\mc{C}_{R'/R}$ is the conductor of $R'$ in $R$. \\
(2) $M$ is a module over $R'$. 
\end{Thm}

\begin{proof}
$(2) \Rightarrow (1)$:  
Follows from Prop.~\ref{Prop:traceconductor} $\tr(M) \subseteq \mc{C}$. \\
$(1) \Rightarrow (2)$: Follows from Prop.~\ref{Prop:tracemodule}.
\end{proof}

\begin{Bem}
In particular, if $R$ is irreducible, then $\tr_R(M)=\mc{C}$ if $M$ is a $\CM$-module over $\widetilde R$. If $R$ is not irreducible, then one can have an inclusion, e.g., consider $R=k[[x,y]]/(xy)$ with $M=R/(x)$. Then $M$ is isomorphic to an irreducible component of the normalization and $\tr_R(M)=(y) \subsetneq \mc{C}=(x,y)$.
\end{Bem}

Looking at the relative conductor $\mc{C}_{R'/R}$ of a ring $R \subseteq R' \subseteq \widetilde R$, we ask if the corresponding statements of Thm.~\ref{Thm:traceequalconductor} also holds for $\CM$-modules $M$ over $R$, namely, whether the following are equivalent: \\
(1) $\tr_R(M) \subseteq \mathcal{C}_{R'/R}$. \\
(2) $M$ is a module over $R'$. 

The implication $(2) \Rightarrow (1)$ holds for any $R$-module with this property by Prop.~\ref{Prop:traceconductor}. However, the other implication is in general not satisfied: consider a non-normal local ring $R$ with canonical module $\omega_R$ such that $\tr_R(\omega_R)=\mc{C}$, the conductor of the normalization (examples for such rings are the non-regular $1$-step normal rings considered in Section \ref{Sec:one-step-normal}). The canonical module is Cohen--Macaulay 
over $R$, but $\omega_R$ is not a module over the normalization $\widetilde R$: by \cite[Thm.~3.3.4]{BrunsHerzog93}  $R \cong \End_R(\omega_R)$ and thus by \cite{DoFI} $\omega_R$ is not a module over $\widetilde R$.

\subsection{Trace ideals and finite CM type} \label{Sub:traceandfiniteCM}

The guiding question of this section is: which ideals in $R$ are isomorphic to trace ideals of Cohen--Macaulay modules? The first case to consider, is $R$ of \emph{finite $\CM$-type}, that is, there are only finitely many isomorphism classes of $\CM$-modules over $R$. Rings of finite $\CM$-type have been extensively studied, in particular, they are classified for rings of Krull-dimension $\leq 2$, see \cite{LeuschkeWiegand} for an overview and references. \\
Here we pose the following

\begin{Qu}
Let $R$ be a complete local or graded ring. Are the following equivalent? \\
(1) $R$ is of finite $\CM$-type. \\
(2) There are only finitely many possibilities for $\tr_R(M)$, where $M \in \CM(R)$.
\end{Qu}

It is clear that $(1)$ implies $(2)$, since $\tr(M_1 \oplus M_2)=\tr(M_1) +\tr(M_2)$ and there are only finitely many isomorphism classes of indecomposable $\CM$-modules over $R$. The other implication is also true for a class of rings coming from geometry: coordinate rings of plane curve singularities. To prove this we will use results of Greuel and Kn\"orrer about $1$-dimensional rings of finite $\CM$-type.

\begin{Prop} \label{Prop:conductorequal}
Let $R$ be a reduced Gorenstein ring. 
 Let $R'$ and $R''$ be two finite birational extensions of $R$. Then $\mc{C}_{R'/R} \cong \mc{C}_{R''/R}$ if and only if $R' \cong R''$.
\end{Prop}

\begin{proof}
If $R'$ is isomorphic to $R''$, then clearly their conductors are isomorphic. For the other direction, if $R \subseteq S$ is finite, then $S$ is a $\CM$-module over $R$. 
Since $R$ is Gorenstein, by Prop.~\ref{Prop:CMimpliesReflexive}, $S$ is also a reflexive $R$-module. Now Theorem \ref{Thm:Conductors-Overrings} implies the result.
\end{proof}

\begin{cor} \label{Cor:finitetraceCM}
Let $(R,\mf{m},k)$ be a reduced local complete ring of dimension $1$ and embedding dimension $2$ and $k$ containing $\QQ$ (that is, $R$ is the coordinate ring of a plane curve singularity). Then $R$ is of finite $\CM$-type if and only if there are finitely many possibilities for $\tr_R(M)$, $M \in \CM(R)$.
\end{cor} 

\begin{proof}
We have already seen that if $R$ is of finite $\CM$-type, then there only exists a finite number of possible trace ideals. Suppose that there are only finitely many possibilities for $\tr(M)$. Since for any possible finite birational extension $R \subseteq S$ one has $\tr_R(S)= \mc{C}_{S/R}$ and by  Prop.~\ref{Prop:conductorequal} $\tr_R(S) \cong \tr_R(S')$ if and only if $S \cong S'$, it follows, that there are only finitely many possible finite birational extensions of $R$. But this is exactly the characterization of coordinate rings plane curves of finite $\CM$-type of Greuel--Kn\"orrer, see \cite{GreuelKnoerrer}, Satz 2.
\end{proof}

\begin{Bem}
The proof of \cite[Satz 2]{GreuelKnoerrer} holds for all Gorenstein complete local rings of dimension $1$ that satisfy $\mathrm{mult}(R) < \mathrm{emb.dim}(R)+2$, so also Cor.~\ref{Cor:finitetraceCM} extends to this case.
\end{Bem}

It is not clear how to generalize this proof to higher dimension, that is, to Gorenstein singularities $X=\Spec(R)$ of finite $\CM$-type and dimension greater than or equal to $2$: by \cite{BuchweitzGreuelSchreyer, KnoerrerCohenMacaulay} these $X$ are precisely the ADE-hypersurface singularities. So in order to get an analog to Cor.~\ref{Cor:finitetraceCM}, one would need to show that any Gorenstein ring $R$ with finitely many isomorphism classes of trace ideals $\CM$-modules is isomorphic to the coordinate ring of one of these ADE-singularities. 

\begin{ex}
An example for a ring of infinite $\CM$-type is the coordinate ring of the swallowtail singularity. The graded rank one $\CM$-modules over the swallowtail were classified by Hovinen \cite[Thm.~4.4.7]{Hovinen}. He showed in particular that the first fitting ideals, or equivalently, the trace ideals, of pairwise non-isomorphic graded $\CM$-modules of rank $1$ are pairwise non-isomorphic. This shows that there are infinitely many possibilities for trace ideals of $\CM$-modules in this example.
\end{ex}

\section{One-step normalization and conductors} \label{Sec:one-step-normal}

In this section, let $(R,\mf{m})$ be a one-dimensional reduced local  ring.
The most natural normalization chain starts with
$$R \subseteq \End_R(\mf{m}) \subseteq \cdots $$
As defined in the introduction, we say that $R$ has a $1$-step normalization or is $1$-step normal if  $\widetilde R \cong \End_R(\mf{m})$. Note that $1$-step normal rings include the regular local rings, since in this case $\mf{m} \cong R$ and thus also $\End_R(\mf{m})\cong R$. Here we characterize $1$-step normal rings in terms of their maximal ideal and study connections with nearly and almost Gorenstein rings, see  \cite{HerzogHibietc} for more on these type of rings.

Recall that a local ring (or positively graded $k$-algebra) $R$ is called \emph{nearly Gorenstein} if $R$ is Cohen--Macaulay, admits a canonical module $\omega_R$ and $\tr_R(\omega) \supseteq \mf{m}$, where $\mf{m}$ is the (graded) maximal ideal of $R$. The trace of the canonical module measures the non-Gorenstein locus of $R$. 
First we prove the following fact, which has been proven for domains in \cite[Prop.~2.14]{CorsoHunekeKatzVasconcelos} (there it is more generally shown that for $1$-dimensional local domains every integrally closed ideal is reflexive):

\begin{Prop} \label{Prop:mreflexive}
Let $(R,\mf{m})$ be a one-dimensional Cohen--Macaulay local ring. Then $\mf{m}$ is a reflexive $\CM(R)$-module.
\end{Prop}

\begin{proof}
First assume that $R$ is regular. Then since the global dimension of $R$ is $1$, the maximal ideal $\mf{m}$ is projective and thus reflexive.
Assume now that $R$ is not regular.  We show that $\mf{m} \cong \mf{m}^{**}=\Hom_R(\Hom_R(\mf{m},R),R)$, which by Prop.~\ref{Prop:reflexivebidual} is enough to show reflexivity of the maximal ideal.
Consider the exact sequence
$$ 0 \lra \mf{m} \lra R \lra R/ \mf{m} \lra 0.$$
Applying $\Hom_R(-,R)$ to it yields the exact sequence
\begin{equation} \label{Eq:mdual}
 0 \lra \Hom_R(R/\mf{m},R) \lra R \lra \mf{m}^* \lra \Ext^1(R/\mf{m},R) \lra 0.
 \end{equation}
Since $R/\mf{m}=k$ is of finite length over $R$ and hence of depth $0$, $\Hom_R(R/\mf{m},R)=0$. Moreover, note that $\Ext^1(R/\mf{m},R)$ is \ef{killed by $\mf{m}$, and thus} a module of finite length over $R$.  
Applying  $\Hom_R(-, R)$ to \eqref{Eq:mdual} gives
$$ 0 \lra \Hom_R(\Ext^1(k,R),R) \lra \mf{m}^{**} \lra R \lra \Ext^1(\Ext^1(k,R), R) \lra \cdots. $$
Here the same reasoning as above implies that $\Hom_R(\Ext^1(k,R),R)=0$ and thus $\mf{m}^{**}$ embeds into $R$. 
Assume that $\mf{m}^{**}$ were isomorphic to $R$. Then, as the dual of any module is reflexive (cf.~Prop.~\ref{Prop:BassMstar-reflexive}), we have $\mf{m}^{*} \cong \mf{m}^{***} = (\mf{m}^{**})^{*} \cong  R$.  Now, looking back at sequence \eqref{Eq:mdual}, the map $R \xrightarrow{} \mf{m}^{*} = R$ must be given by multiplication by some $a \in R$. But since the next term is killed by $\mf{m}$,  $R/aR$ is killed by $\mf{m}$, which implies that $aR=\mf{m}$ and  $R$ is a discrete valuation ring, contradiction. \\
On the other hand, $\Ext^1(\Ext^1(k,R), R)$ is of finite length and annihilated by $\mf{m}$, which implies that the image of $R$ in this module is also annihilated by $\mf{m}$. So the image is a factor of $R/\mf{m}=k$ and since the image is non-zero, it must equal $k$. But this implies that $\mf{m}^{**}$ is isomorphic to $\mf{m}$.
\end{proof}

\begin{Prop} \label{Prop:conductor-1-step}
 Let $(R,\mf{m})$ be a reduced one-dimensional noetherian local ring. Then $\mc{C} \supseteq \mf{m}$ if and only if $R$ is $1$-step normal.
\end{Prop}

\begin{proof} 
First assume that $R$ is $1$-step normal. If $R$ is regular, then $\mc{C}=R$ and we do not have to show more. If $R$ is not regular, then it is easy to see (by direct calculation) that $\End_R(\mf{m})=\Hom_R(\mf{m},R)$. 
Since $R$ has a $1$-step normalization, this means that $\mf{m}^* \cong \widetilde R$. Applying $\Hom_R(-,R)$ to both sides of the equation yields
$$\mf{m}^{**} \cong (\widetilde R)^*\cong \mc{C} \ .$$
Since by Prop.~\ref{Prop:mreflexive} $\mf{m}^{**}$ is isomorphic to $\mf{m}$, we have that $\mf{m} \cong \mc{C}$. It remains to show that $\mc{C} = \mf{m}$.  Clearly $\mc{C} \subseteq \mf{m}$. For the other inclusion we only need to show that $\mf{m}$ is closed under multiplication in $\widetilde R$, which will render $\mf{m}$ an ideal of $\widetilde R$ and then by definition of the conductor, $\mf{m} \subseteq \mc{C}$. Therefore, let $\varphi: \mc{C} \xrightarrow{} \mf{m}$ be the isomorphism (as $R$-modules). Any element in $\widetilde R$ can be written as $\frac{r}{s}$, for some $r,s \in R$, where $s$ is a nonzerodivisor. Let $m \in \mf{m}$, then $m=\varphi(c)$ for some $c \in \mc{C}$. Since $\mc{C}$ is closed under multiplication in $\widetilde R$, $\frac{r}{s}c=c'$ for some $c' \in \mc{C}$, or equivalently, $rc=sc'$. Applying the $R$-isomorphism $\varphi$ yields that $r\varphi(c)=rm=s\varphi(c')$ is in $\mf{m}$. Since $s$ is a nonzerodivisor and $\varphi(c') \in \mf{m}$, the claim follows.\\
For the other implication assume that $\mc{C}\supseteq  \mf{m}$. If the inclusion is strict, then $R$ is regular and hence $1$-step normal. If $\mf{m}=\mc{C}$, then $\End_R(\mc{C})= \End_R(\mf{m})$, and since the left hand side is equal to the normalization $\widetilde R$, our claim follows.
\end{proof}

We have the following connection between nearly Gorenstein rings and 1-step normal rings:

\begin{cor} \label{Cor:1-step-nearly-Gorenstein}
Under the conditions of the proposition, if $R$ is $1$-step normal, then $R$ is nearly Gorenstein.
\end{cor}

\begin{proof}
 Since either $R$ is regular or $\mc{C}=\mf{m}$, by Prop.~6.6 of \cite{HerzogHibietc} $R$ is nearly Gorenstein.
\end{proof}

\begin{Bem}
 (1) The other implication does not hold: by Prop.~7.1 of \cite{HerzogHibietc}, the semigroup ring $R=k[[t^3,t^5,t^7]]$ is nearly Gorenstein. But the endomorphism ring of the maximal ideal $\mf{m}=(t^3,t^5,t^7)$ is $k[[t^2,t^3]]$, which is clearly not the normalization of $R$. \\
 (2) There is also the slightly different concept of \emph{almost Gorenstein rings}, that was first considered by Barucci--Fr\"oberg and later generalized by \cite{GotoTakahashiTaniguchi}. It can be shown that $1$-dimensional almost Gorenstein rings are always nearly Gorenstein (see \cite[Prop.~6.1]{HerzogHibietc}) but one can easily find examples of nearly Gorenstein rings that are not almost Gorenstein. There seem to be no clear relation between almost Gorenstein rings and $1$-step normalization rings, since e.g., the ring $R$ from (a) is also almost Gorenstein but not $1$-step normal and on the other hand the ring $S=k[[t^5,t^6,t^7]]$ is $1$-step normal but not almost Gorenstein (see \cite[Remark 6.2]{HerzogHibietc}.
\end{Bem}

Another interesting property of $1$-step normal rings is that the maximal ideal is isomorphic to its dual.

\begin{Prop} \label{Prop:onestepm}
Let $(R,\mf{m})$ be a $1$-step normal complete integral domain of dimension $1$. Then $\mf{m} \cong \mf{m}^{*}$.
\end{Prop}

\begin{Bem} The following short proof was kindly pointed out by the referee. Alternatively, one could use more representation theoretic methods and study the endomorphism ring of the $R$-module $M=R\oplus \widetilde R$ to show that the maximal ideal $\mf{m} \in \add M$. Then either $R$ is regular and $\mf{m} \cong R \cong \mf{m}^*$, or otherwise this implies that $\mf{m}\cong \widetilde R$.
\end{Bem}

\begin{proof}
If $R$ is regular, then $\mf{m} \cong R \cong \mf{m}^*$. Assume that $R$ is $1$-step normal, but not regular. Then $\widetilde R\cong \End_R(\mf{m})=\Hom_R(\mf{m},R)=\mf{m}^* \supsetneq R$. We show that then $\mf{m}$ itself is isomorphic to $\widetilde R$, which implies the claim. For this note that by Prop.~\ref{Prop:conductor-1-step} the conductor $\mc{C}=\mf{m}$, and thus $\mf{m}\widetilde R \subseteq R$. One can easily see  that the units of $R$ are not contained in $\mf{m}\widetilde R$. So this inclusion is strict, which implies $\mf{m}\widetilde R \subseteq \mf{m}R$. On the other hand, $\mf{m}R \subseteq \mf{m}\widetilde R$, so it follows that $\mf{m}R=\mf{m}\widetilde R$. Since $R$ is an integral domain, the normalization $\widetilde R$ is a discrete valuation ring, and the ideal $\mf{m}\widetilde R$ must be principal and isomorphic to $\widetilde R$. So $\mf{m}$ is indeed isomorphic to $\widetilde R$.
\end{proof}

\begin{cor}  \label{Cor:maximal-isom-normalization}
Let $(R,\mf{m})$ be a one-dimensional complete integral domain. Then $R$ is $1$-step normal if and only if $\widetilde R$ is isomorphic to $\mf{m}$.
\end{cor}

\begin{proof}
If $R$ is $1$-step normal then the proof of Prop.~\ref{Prop:onestepm} shows that $\mf{m} \cong \widetilde R$. Then 
$$\End_R(\mf{m})\cong \End_R(\widetilde R) = \widetilde R.$$
The other implication is clear.
\end{proof}

\begin{Bem}
Note that there is also the notion of \emph{self-dual module}: if $M$ is a module over local ring $R$ with canonical module $\omega$, then $M$ is self-dual if $M\cong \Hom_R(M,\omega)$. The property of $\mf{m}$ to be self-dual in this sense has e.g.~been appeared in \cite{HunekeVraciu, Ooishi} and more recently been studied in \cite{Kobayashi}.
\end{Bem}

One can also ask if in Prop.~\ref{Prop:onestepm} the condition $\mf{m} \cong \mf{m}^*$ is also sufficient for $R$ to be $1$-step normal, but this is certainly not true:

\begin{ex} \label{Ex:selfdualmax}
Let $R$ be a reduced local ring of dimension one with infinite residue field and not a discrete valuation ring. Assume further that the  multiplicity of $R$ is equal to $2$. By \cite[Thm.~2.6]{Ooishi} this condition is equivalent with the property that the embedding dimension of $R$ is $2$ and $\mf{m}=\mf{m}^{*}$ (and even that any $\CM$-module over $R$ is self-dual). Thus any plane curve of multiplicity two has the property $\mf{m} \cong \mf{m}^{*}$. Explicit examples are coordinate rings of $A_n$-curves, which do not have a one-step normalization for $n \geq 3$.  
\end{ex}

\section{Global spectrum of curve singularities}

In this final section we assume that $R$ is complete local noetherian. We are interested in the global spectrum $\gs(R)$ of $R$, where $R$ is one-dimensional and reduced (i.e., $\Spec(R)$ is a curve singularity), defined as 
$$\gs (R)=\{ n \in \NN : \textrm{ there exists a module } M \in \CM(R), \textrm{ such that } \gl \End_R(M) = n \} \ .$$
One restricts to $\CM$-modules because this is a class of modules for which one can use methods from representation theory.  In order to show that $1$ and $2$ are always in the global spectrum (Lemma \ref{12ingspec}) and characterize when $3$ is in the global spectrum (Theorem \ref{Thm:3ingspec}, under some additional assumptions on $R$) we will use normalization chains.

We start with a reduced local complete noetherian ring $R$ of dimension $1$ and consider a chain of the form \eqref{Eq:chain}. We want to determine 
\begin{equation} \label{Eq:endchain}
\gl \End_R(M), \textrm{ where }\quad  M=\bigoplus_{j=0}^l R_j \ .
\end{equation}
The most natural chain to study is
\begin{equation} \label{Eq:maxchain}
R=R_0 \subseteq \End_R(\mf{m}_R)= R_1 \subseteq \cdots \subseteq \End_{R_{l-1}}(\mf{m}_{R_{l-1}})=R_l = \widetilde R \ ,
\end{equation}
as in \cite{Leuschke07}. Also cf.~\cite{IyamaRejective} for a more representation theoretic approach (here the non-regular one dimensional rings are considered as orders over a DVR, which can be taken to be the Noether normalization since we assume $R$ to be complete). It has been shown that the global dimension of such an $\End_R(M)$ is bounded above by $l$ (see \cite[Thm.~4]{Leuschke07}), but it is not clear which value it will assume in general. Let us note that such chains and the global dimension of rings of the form \eqref{Eq:endchain} have been studied for semigroup rings   by Mousavidehshikh: he has shown that for a given integer $n$, one can construct certain semigroup rings $R$ such that $n \in \gs(R)$, and moreover that always $2 \in \gs(R)$ for these $R$, see \cite[Thm.~5.4 and Thm.~5.9]{Mousavidehshikh}. In general it is not clear, which integers are contained in $\gs(R)$.

\subsection{Elements in the global spectrum}
\begin{lem} \label{12ingspec}
Let $R$ be a complete local CM ring of Krull-dimension one, which is not regular. Then $\{1,2\}$ is always a subset of the global spectrum, that is, there always exist NCRs of global dimension $1$ and $2$ of $\Spec(R)$.
\end{lem}

\begin{proof}
For a ring $R$ with these properties the normalization $\widetilde R$ is a finitely generated $\CM$-module. Since $\widetilde R$ is regular, its global dimension is $1$ and from $\End_R \widetilde R=\End_{\widetilde R} \widetilde R=\widetilde R$ follows that $1 \in \gs (R)$. Now look at the chain \eqref{Eq:maxchain}.
If $l=0$, then $R$ is regular. So we have $l \geq 1$. Let $M=R_{l-1}\oplus R_l$. Then, since all $R_i$'s are finite extensions of $R$, one has $\End_R (M)=\End_{R_i}(M)$ for all $i=1, \ldots, l-1$ (see Lemma \ref{Lem:HomChangeofRings}). Thus we may assume w.l.o.g. that $l=1$, i.e., $M=R_0 \oplus R_1$, where $R_0=R$ and $R_1=\widetilde R$. Then $\End_R M$ is a noncommutative resolution which is a generator. By  \cite[Prop.~2.8]{DFI} the global dimension is strictly greater than $1$. By \cite{Leuschke07} Theorem 4, the global dimension is bounded by $2$. Hence $\gl \End_R M=2$. 
\end{proof}

This lemma immediately shows that for a complete local reduced ring $R$ of Krull-dimension $1$, one has  $\gs(R)=\{1\}$ if and only $R$ is regular. For rings of Krull-dimension $2$ a slightly weaker statement holds: if $R$ is a singular complete normal domain with residue field of characteristic $0$, then $\gs(R)=\{2\}$ if and only if $\Spec(R)$ is a simple singularity, cf.~\cite[Cor.~4.13]{DFI}. For a ring of Krull-dimension $d \geq 3$ it is not clear how to interpret the equality $\gs(R)=\{d\}$ in terms of singularities of $\Spec(R)$. \\

For the remaining results, we make use of the classification of the complete equicharacteristic $0$ curve singularities of finite $\CM$ type by Greuel--Kn\"orrer \cite{GreuelKnoerrer}. Here we say that $X=\Spec(R)$ is a \emph{complete equicharacteristic $0$ curve singularity} if $(R,\mf{m})$ is a complete local CM ring of Krulldimension $1$ and $k=R/\mf{m}$ is algebraically closed and of characteristic $0$. For the general classification of one-dimensional local rings of finite $\CM$ type, a good overview can be found in \cite{LeuschkeWiegand}.

\begin{Prop} \label{Prop:finiteCM}
Let $X=\Spec(R)$ be a complete equicharacteristic $0$ curve singularity with $R$ of finite $\CM$ type. Then $3 \in \gs (R)$ if and only if $\Spec(R)$ is not regular or an $A_{2n}$-singularity.
\end{Prop}

\begin{proof}
First note that the only Gorenstein curves of finite $\CM$-type are the simple plane curves, see \cite[Korollar 1]{GreuelKnoerrer}. The non-Gorenstein curves of finite $\CM$-type are the ones birationally dominating the simple plane curves, see e.g. \cite{Yos}. 
Let $R$ be any of the finite $\CM$ type rings. Then from \cite[Thm.~4.6]{DoFI} it follows that $3$ is contained in the global spectrum if and only if $R$ is not an $A_{2n}$-singularity.
\end{proof}

\begin{Prop} \label{Prop:onestep}
Let $X=\Spec(R)$ be a complete equicharacteristic $0$ curve singularity and assume that $R$ is $1$-step normal. Then if $R$ is Gorenstein, $\Spec(R)$ is either regular with $\gs(R)=\{1\}$, an $A_1$-singularity with  $\gs (R)=\{1,2,3\}$ or an $A_2$-singularity with $\gs (R)=\{1, 2\}$. If $R$ is not Gorenstein, then there exists a generator/cogenerator $M$ such that $\gl \End_R(M)=3$.
\end{Prop}

\begin{proof} \ef { If $R$ is regular, then the only indecomposable $\CM$-module is $R$ itself and $\gl \End_R (R)=\gl R=1$, which implies that $\gs (R)=\{ 1 \}$. Therefore assume that $R$ is not regular.}
We will consider $R$ as an order over the DVR $T$ as e.g., in \cite{Dieterich, IyamaConstanta}. Since $R$ is complete, it has a canonical module $\omega_R$, which is given as $\Hom_T(R,T)$. Consider $M=R \oplus \widetilde R$, then by Leuschke's theorem \cite{Leuschke07}, the global dimension of $\End_R(M)\leq 2$. It is equal to $2$ by the same reasoning as in the proof of lemma \ref{12ingspec}. If $R$ is of finite $\CM$ type then Prop.~\ref{Prop:finiteCM} shows that $3 \in \gs (R)$ if and only if $\Spec(R)$ is not of type $A_{2n}$. If $R$ is of infinite $\CM$-type, then by Prop.~\ref{Prop:conductor-1-step}, the maximal ideal $\mf{m}$ is equal to the conductor of the normalization, and hence $\mf{m}$ is an ideal in $\widetilde R$. Then, since $\widetilde R$ is regular and thus of finite $\CM$ type, by Thm.~4.4 of \cite{IyamaRejective} (generalization of \cite{ErdmannHolmIyamaSchroeer}, Thm.~1.1) the representation dimension of $R$ is bounded above by $3$. Since $R$ is not of finite $\CM$-type it is equal to $3$ by \cite[Thm.~4.1.3]{IyamaRejective}.  Explicitly, one may take $M=R \oplus \omega_R \oplus \widetilde R$, by \cite[Thm.~4.1.3]{IyamaRejective} $\gl \End_R M \leq 3$ and it is equal to two if and only if $\add M=\CM (R)$.

\end{proof}

\begin{Thm} \label{Thm:3ingspec}
Let $X=\Spec(R)$ be a complete equicharacteristic $0$ curve singularity. Then $3 \in \gs(R)$ if and only if $\Spec(R)$ is neither regular nor an $A_{2n}$-singularity, $n \geq 1$. Equivalently: $\Spec(R)$ is an $A_{2n}$-singularity if and only if $\gs(R)=\{1,2\}$.
\end{Thm}

\begin{proof}
One direction is clear: if $R$ is regular, then $\gs (R)=\{ 1 \}$, as argued in Prop.~\ref{Prop:onestep}. The global spectrum of an $A_{2n}$-singularity is $\gs R=\{1, 2\}$, by \cite[Thm.~4.4]{DoFI}. 

For the other direction
use the same notation as in the proof of Lemma \ref{12ingspec} above. 
If $l=1$ then the assertion follows from Prop.~\ref{Prop:onestep}.
Now assume that $l \geq 2$. Then there is a chain $R \subseteq R_1=\End_R \mf{m} \subseteq \cdots \subseteq R_l=\End_{R_{l-1}}\mf{m}_{R_{l-1}}=\widetilde R$. Suppose that $3 \not \in \gs (R)$. Argue by induction: for $l=1$ we have shown the assertion. Now we may assume that $R_{l-i}$ for all $i=0, \ldots, k \leq  l$ are $A_{2i}$-singularities. Then $R_{l-i-1} \hookrightarrow R_{l-i}$ is a radical embedding, see \cite{ErdmannHolmIyamaSchroeer} for the definition of this term. If $R_{l-i-1}$ were not of finite $\CM$-type then Theorem \cite[Thm~4.4]{IyamaRejective} 
 would yield an endomorphism ring of global dimension $3$. If $R_{l-i-1}$ is of finite $\CM$-type, then Prop.~\ref{Prop:finiteCM} shows that the only possibility is an $A_{2i+2}$-singularity.
\end{proof}

\begin{Bem}
It would be interesting to consider also chains of endomorphism rings that involve endomorphism rings over modules of rank $\geq 2$, in particular, one could make use of trace ideals in order to find suitable modules $M$ that are defined on a large integral extension of $R$.
\end{Bem}

\section{Acknowledgements}
The author wants to thank the anonymous referee for their helpful comments, additional references, and suggestions to improve the paper. \\
The work in this paper benefited from discussions with Ragnar Buchweitz. The author is deeply thankful for his advice - as well as for many discussions about Mathematics, Life, the Universe, and Everything. \\


\newcommand{\etalchar}[1]{$^{#1}$}
\def\cprime{$'$} \def\cprime{$'$}

\end{document}